\documentclass[letterpaper, 10 pt, conference]{ieeeconf}
\usepackage[utf8]{inputenc}
\usepackage{amsmath}
\usepackage{amsfonts}
\usepackage{amssymb}
\usepackage{cite}
\usepackage{hyperref}
\usepackage{subcaption}
\let\labelindent\relax
\usepackage{enumitem}   
\usepackage{graphicx}   
\usepackage{color}
\usepackage{graphicx}
\usepackage{float}
\usepackage{comment}
\usepackage{tikz}
\usepackage{pgfplots}
\usepackage{standalone}
\usepackage[nospread, noshrink]{cuted}
\usepackage{multicol}
\usepackage{dblfloatfix} 
\usepackage{afterpage}

\usepackage{amsthm}
\newtheoremstyle{ieeeconf}
{0pt}   
{0pt}   
{\normalfont}  
{\parindent}       
{\itshape} 
{:}         
{ } 
{\thmname{#1} \thmnumber{#2}\thmnote{ (#3)}} 
\makeatletter
\renewenvironment{proof}[1][\proofname]{\par
\pushQED{\qed}%
\normalfont \topsep\z@
\trivlist
\item[\hskip2em
\itshape
#1\@addpunct{:}]\ignorespaces
}{%
\popQED\endtrivlist\@endpefalse
}
\makeatletter

\theoremstyle{ieeeconf}
\newtheorem{theorem}{Theorem}   
\newtheorem{lemma}{Lemma}
\newtheorem{problem}{Problem}

\newtheorem{example}{Example}
\newtheorem{remark}{Remark}

\newtheorem{definition}{Definition}

\DeclareMathOperator{\im}{im}

\newcommand{\R}{\mathbb{R}}

\newcommand{\N}{\mathbb{N}}
\newcommand{\D}{\mathcal{D}}

\newcommand{\B}{\mathcal{B}}

\newcommand{\set}[2]{\left\{ #1 \;\middle| \; #2 \right\}}
\usetikzlibrary{shapes,arrows}
\usetikzlibrary{arrows,calc,positioning}
\tikzset{
block/.style = {draw, rectangle,
minimum height=1.5cm,
minimum width=3cm},
input/.style = {coordinate,node distance=5cm},
output/.style = {coordinate,node distance=2cm},
arrow/.style={draw,-latex, node distance=5cm},
sum/.style = {draw, circle, node distance=1cm},
}
\IEEEoverridecommandlockouts
\begin{document}
\title{\LARGE \bf A Generalized Stability Theorem for Interconnections of Dissipative Behaviors}
\author{Michaela Repášová and Henk J. van Waarde
\thanks{Michaela Repášová and Henk J. van Waarde are with the Bernoulli Institute for Mathematics, Computer Science, and Artificial Intelligence, University of Groningen, Nij\-enborgh 9, 9747 AG, Groningen, The Netherlands. Email: {\tt\small m.repasova@rug.nl,h.j.van.waarde@rug.nl}.
}%
\thanks{
	Henk J. van Waarde acknowledges financial support by the Dutch Research Council (NWO) under the Talent Programme Veni Agreement (VI.Veni.22.335).
}%
}

\maketitle
\thispagestyle{empty}
\pagestyle{empty}

\begin{abstract}
Dissipativity theory enables the stability analysis of complex, interconnected systems through the physical properties of their components. Celebrated examples of such stability results include the small gain and passivity theorems. In this paper, we prove a unifying stability theorem for dissipative linear time-invariant (LTI) systems. We use the behavioral approach to express storage functions and supply rates in terms of quadratic differential forms (QDFs). Our main result establishes conditions for the stability of an interconnection of two LTI behaviors that are both dissipative with respect to general, dynamic supply rates. As special cases, we recover existing results for static supply rates, such as those capturing finite $L_2$-gain and passivity, as well as dynamic ones capturing, for example, delta dissipativity. 
The key idea of the paper is to construct a Lyapunov function as a sum of storage functions \emph{and} a coupling QDF, thereby allowing for a richer class of Lyapunov functions than existing works dealing with just sums of storage functions.
\end{abstract}

\section{Introduction}
Dissipativity is one of the cornerstones of systems and control theory. It was introduced by Jan Willems in the 1970's \cite{willems_dissipative_1972,willems_dissipative_1972-1}, as a general type of input-output stability. A system is said to be dissipative if the amount of energy stored at a given time cannot exceed the amount that has been supplied to the system in the past. This stored energy is captured by a storage function, while the energy supplied to the system is determined by a supply rate. 
Supply rates can be either \emph{static} or \emph{dynamic}, depending on whether the supplied energy is a static function of the external variables or also depends on their derivatives.
What makes dissipativity so useful is that it allows the (stability) analysis of interconnected systems through the
properties of their individual components.

Famous examples of stability theorems based on dissipativity include the small gain and passivity theorems. Early works in this direction date back to contributions by Zames in \cite{zames_input-output_1966}, and, by now, several versions of these results exist for different system descriptions such as input-output operators and state-space models \cite{desoer_feedback_2009,van_der_schaft_dissipative_2017}.  
Moreover, various extensions of these stability theorems exist, for example, to general static supply rates \cite{arcak_stability_2016,van_der_schaft_dissipative_2017}.  
In \cite{xiong_negative_2010}, the authors have shown conditions under which interconnections of so-called negative imaginary systems are stable. One can view the negative imaginary property as a type of dissipativity with a dynamic supply rate where a derivative of the output appears. Yet another example of dissipativity with respect to a dynamic supply rate is delta dissipativity, involving a derivative in both the input and output, for which a stability result is given in \cite{schweidel_compositional_2022}. 

Despite these research efforts, the literature on the stability of system interconnections is seemingly disconnected. Indeed, for each of the individual supply rates, the stability results were proven independently and with different system descriptions. There is no overarching theory explaining all these results. It appears that the inclusion of dynamics in the supply rates affects the stability in complex ways. In fact, for the small gain and passivity theorems, an established idea is to prove stability of the interconnection by means of a Lyapunov function that is a sum of storage functions of the individual system components. However, as we will show, this is not always possible for interconnections of negative imaginary systems. Nonetheless, all of the aforementioned stability results have much in common. They consist of conditions on the external behavior of dissipative systems under which their interconnection is stable. It is of interest to find what unites these results, and to develop a general theory applicable to systems that are dissipative with respect to general dynamic supply rates.

In this paper, we adopt the behavioral approach \cite{willems_time_1986I,willems_time_1986II,polderman_introduction_1998} to systems and control. We thus view systems as sets of trajectories, and define dissipativity using the concept of quadratic differential forms (QDFs) \cite{willems_quadratic_1998} to formalize both storage functions and supply rates. We believe behavioral theory is ideally suited for studying the stability of system interconnections because it allows for general system interconnections beyond feedback \cite{willems_interconnections_1997}, and naturally incorporates general dynamic supply rates as QDFs.

This paper contains two stability results for an interconnection of two dissipative linear time-invariant (LTI) behaviors. Our first theorem (Theorem~\ref{t: first diss inter thm}) states that an interconnection of two dissipative behaviors is (asymptotically) stable if there exists a weighted sum of the supply rates that is nonpositive. This can be viewed as the behavioral extension of the established idea of defining a Lyapunov function as a sum of storage functions. Note that such an extension was also studied in \cite{takaba_stability_2005}, under different technical assumptions; see Remark \ref{r: another behavioral}. 
Theorem~\ref{t: first diss inter thm} generalizes many results like the small gain theorem, passivity theorem, and an interconnection result for delta dissipative systems. However, it is not applicable to negative imaginary systems. This motivates our main contribution: a general, unifying stability theorem for system interconnections (Theorem~\ref{t: main diss inter thm}). It states that the interconnection of two dissipative behaviors is (asymptotically) stable if there exists a coupling QDF depending on the supply rates such that the weighted sum of storage functions minus the coupling QDF is nonnegative. With this result, we provide a more general class of Lyapunov functions that, in addition, allows for indefinite storage functions. 

\emph{Outline}: In Section \ref{sec:Preliminaries}, we introduce preliminaries of the behavioral theory of dynamical systems, quadratic differential forms, stability, and dissipativity. In Section \ref{sec:problem statement}, we introduce the problem, and in Section \ref{sec:first theorem}, we state and prove our first result and provide some examples. Lastly, in Section \ref{sec:main theorem}, we state and prove our main result and provide additional examples.

\emph{Notation}:
We denote by ${\mathcal{C}^\infty(\R,\R^q)}$ the space of infinitely differentiable functions mapping from $\R$ to $\R^q$. We define
\begin{equation*}
\mathcal{D}(\R,\R^q):=\set{w \in \mathcal{C}^\infty(\R,\R^q)}{w \text{ has compact support}}.
\end{equation*}We denote by $\mathbb{W}^{\mathbb{T}}$ the collection of all maps from $\mathbb{T}$ to $\mathbb{W}$. 
Moreover, ${\R^{p \times q}[\xi]}$ denotes the set of real polynomial matrices of size $p \times q$ in the indeterminate $\xi$.
Let $f : \R \to \R^q$ and $K \in \R^{r \times q}$. We denote by $Kf$ the function $t \mapsto Kf(t)$ and, if $r=q$, by $f^\top Kf$ the function $t\mapsto f(t)^\top Kf(t)$.
\section{Preliminaries}
\label{sec:Preliminaries}
\subsection{Behavioral theory}
In this section, we recall a few definitions from the behavioral theory of linear time-invariant dynamical systems. 

\begin{definition}[\hspace{-0.1ex}{\cite[Definition~1]{willems_time_1986I}}]
A \emph{dynamical system} is defined as a triple $\Sigma=(\mathbb{T}, \mathbb{W},\B)$ with $\mathbb{T} \subseteq \R$ the \emph{time set}, $\mathbb{W}$ the \emph{signal alphabet}, and $\B \subseteq \mathbb{W}^{\mathbb{T}}$ the \emph{behavior} of the system. 
\end{definition}

In this paper, we consider linear time-invariant (LTI) dynamical systems of the form $\Sigma=(\R,\R^q,\B)$ where $\B$ is given by 
\vspace{-1.5ex}
\begin{equation}
\B = \set{w\in \mathcal{C}^\infty(\R,\R^q)}{R\left(\frac{d}{dt}\right)w=0},
\label{behavior}
\end{equation} 
for some $R \in \R^{p \times q}[\xi]$. 
We denote by $\mathcal{L}^q$ the set of all behaviors $\B$ of the form \eqref{behavior}. Note that in order to avoid mathematical technicalities, we focus on infinitely differentiable trajectories. However, all of the results of this paper can be extended to behaviors defined in terms of locally integrable trajectories of LTI systems.

Let $\B \in \mathcal{L}^q$
and  $R \in \R^{p \times q}[\xi]$. We call \vspace{-0.5ex}
\begin{equation}
R\left(\frac{d}{dt}\right)w=0
\label{ker rep}
\end{equation} a \emph{kernel representation of} $\B$ if \eqref{behavior} holds. In this case, we use the notation $\B =\ker R\left(\frac{d}{dt}\right)$. 

Lastly, we discuss the controllability of behaviors.
\begin{definition}[\hspace{-0.01ex}{\cite[Definition~5.2.2]{polderman_introduction_1998}}]
A behavior $\B \in \mathcal{L}^q$ is \emph{controllable} if for each $w_1,w_2 \in \B$ there exists a $t_1 \geq 0$ and a trajectory $w\in \B$ such that $w(t)=w_1(t)$ for $t\leq0$ and ${w(t)=w_2(t-t_1)}$ for $t\geq t_1$.
\end{definition} 
Let $M \in \R^{q \times r}[\xi]$. We call 
\vspace{-1.5ex}
\begin{equation}
w=M \left(\frac{d}{dt}\right)l
\label{image rep}
\end{equation}
an \emph{image representation of} $\B$ if 
\begin{equation}
\B=\set{ M\left(\frac{d}{dt}\right)l}{ l\in \mathcal{C}^\infty(\R,\R^r)}.
\label{B in image form}
\end{equation}
In this case, we use the notation $\B =\im M\left(\frac{d}{dt}\right)$. 
Controllable systems are exactly those that admit image representations \cite{polderman_introduction_1998}. That is, $\B \in \mathcal{L}^q$ is controllable if and only if there exists an $r\in \N$ and an $M \in \R^{q \times r}[\xi]$ such that \eqref{B in image form} holds.

\subsection{Quadratic Differential Forms}
In this section, we discuss quadratic differential forms (QDFs) as introduced in \cite{willems_quadratic_1998}.

Let $\R_s^{q \times q}[\zeta,\eta]$ denote the set of real symmetric polynomial matrices in the (commuting) indeterminates $\zeta$ and $\eta$. Explicitly, an element $\Phi 
\in \R_s^{q \times q}[\zeta,\eta]$ is given by 
\begin{equation}
\Phi(\zeta,\eta)=\sum_{k,l=0}^N\Phi_{kl}\zeta^k\eta^l,
\label{2var pol def}
\end{equation}
where $\Phi_{kk}\in \R^{q \times q}$ are symmetric matrices and $\Phi_{kl}=\Phi_{lk}^\top$ for all ${k,l \in \{0,1,\dots,N\}}$. We call the smallest $N$ for which \eqref{2var pol def} holds
the \emph{order} of $\Phi$. Such $\Phi$ induces a so-called \emph{quadratic differential form} (QDF)
\begin{equation*}
Q_\Phi : \mathcal{C}^\infty (\R,\R^{q}) \to \mathcal{C}^\infty (\R,\R),
\vspace{-1ex}
\end{equation*}
defined by 
\vspace{-1ex}
\begin{equation*}
\left( Q_\Phi(w)\right)(t)=\sum_{k,l=0}^N \left(\frac{d^kw}{dt^k}(t)\right)^\top \Phi_{kl} \left(\frac{d^lw}{dt^l}(t)\right).
\end{equation*}
We associate with the QDF $Q_\Phi$ 
the coefficient matrix ${\tilde{\Phi}\in \R^{(N+1)q \times (N+1)q}}$ with block entries $\Phi_{ij}$
such that \vspace{-0.5ex}
\begin{equation*}
\left(Q_\Phi(w)\right)(t) =  \begin{bmatrix}
	w(t) \\
	\frac{dw}{dt}(t) \\ \vdots \\ \frac{d^Nw}{dt^N}(t)
\end{bmatrix}^\top \tilde{\Phi} \begin{bmatrix}
	w(t) \\
	\frac{dw}{dt}(t) \\ \vdots \\ \frac{d^Nw}{dt^N}(t)
\end{bmatrix}.
\end{equation*}
We say that the QDF $Q_\Phi$ is of \emph{order} $N$ if the polynomial matrix $\Phi$ inducing the QDF is of order $N$. 

Next, we discuss properties of QDFs. Let $\Phi \in \R_s^{q \times q}[\zeta,\eta]$ and consider the QDF $Q_\Phi$. 
Its derivative, denoted by $\frac{d}{dt}Q_\Phi$, is also a QDF. That is, $\frac{d}{dt}Q_\Phi = Q_{\dot{\Phi}}$,
where \[{\dot{\Phi}(\zeta,\eta)=(\zeta+\eta)\Phi(\zeta,\eta)}.\]

Now we define the nonnegativity and positivity of QDFs. 
\begin{definition}
	Consider $\Phi \in \R_s^{q \times q}[\zeta, \eta]$. We call the QDF $Q_\Phi$ \emph{nonnegative}, denoted by $Q_\Phi \geq 0$, if ${Q_\Phi(w) \geq 0}$ for all $w\in \mathcal{C}^\infty(\R,\R^q)$, that is, ${\left(Q_\Phi(w)\right)(t) \geq 0}$ for all ${w\in \mathcal{C}^\infty(\R,\R^q)}$ and all $t\in \R$. Moreover, we call $Q_\Phi$ \emph{positive}, denoted by $Q_\Phi>0$, if ${Q_\Phi \geq 0}$ and if the only ${w\in \mathcal{C}^\infty (\R,\R^q)}$ for which ${Q_\Phi(w)=0}$, that is, ${\left(Q_{\Phi}(w)\right)(t)=0}$ for all $t\in \R$, is $w=0$. 
\end{definition}
These notions induce nonpositivity and negativity. Namely, we call the QDF $Q_\Phi$ \emph{nonpositive}, denoted by $Q_\Phi \leq 0$, if $Q_{-\Phi} \geq 0$ and \emph{negative}, denoted by $Q_\Phi <0$, if $Q_{-\Phi} >0$.

Moreover, we define these notions on a behavior $\B \in \mathcal{L}^q$.
\begin{definition}
	Consider $\Phi \in \R_s^{q \times q}[\zeta, \eta]$ and $\mathcal{B}\in \mathcal{L}^q$. We call the QDF $Q_\Phi$ \emph{zero} on $\mathcal{B}$, denoted by $Q_\Phi \stackrel{\mathcal{B}}{=}0$, if \\$Q_\Phi(w)=0$ for all $w\in \mathcal{B}$. Moreover, we call it \emph{nonnegative} on $\B$, denoted by $Q_\Phi \stackrel{\mathcal{B}}{\geq} 0$, if $Q_\Phi(w) \geq 0$ for all $w\in \B$, and \emph{positive} on $\B$, denoted by $Q_\Phi\stackrel{\mathcal{B}}{>}0$, if $Q_\Phi \stackrel{\mathcal{B}}{\geq} 0$ and if the only $w\in\B$ for which $Q_\Phi(w)=0$ is $w=0$.
\end{definition}
Nonpositivity and negativity on $\B$ are defined analogously, and are denoted by $Q_\Phi \stackrel{\B}{\leq}0$ and $Q_\Phi \stackrel{\B}{<}0$, respectively.
\subsection{Stability}
In this section, we discuss stability in the context of behaviors. To do so, we first recall the definition of autonomous behaviors.

\begin{definition}[\hspace{-0.01ex}\cite{willems_quadratic_1998}] The behavior $\B\in \mathcal{L}^q$ is \emph{autonomous} if for any $w_1,w_2 \in \mathcal{B}$ such that $w_1(t)=w_2(t)$ for $t<0$ we have that $w_1(t)=w_2(t)$ for all $t \in \R$.
\end{definition} 

\begin{definition}[\hspace{-0.01ex}\cite{willems_quadratic_1998}] Consider an autonomous behavior ${\B \in \mathcal{L}^q}$. We call $\B$ \emph{stable} if every $w\in \B$ is bounded on the half-line $[0, \infty)$ and \emph{asymptotically stable} if ${\lim_{t \to \infty}w(t)=0}$ for all $w\in \mathcal{B}$. 
\end{definition}

Lastly, we recall the following conditions for (asymptotic) stability in terms of Lyapunov functions described by QDFs.
\begin{theorem}[\hspace{-0.01ex}\cite{willems_quadratic_1998}]
	Let ${\B \in \mathcal{L}^q}$ be autonomous.
	\begin{enumerate}[label=\roman*)]
		\item  $\B$ is stable if and only if there exists $\Psi \in \R_s^{q \times q}[\zeta, \eta]$ such that $Q_\Psi \stackrel{\mathcal{B}}{>} 0$ and $\frac{d}{dt}Q_\Psi \stackrel{\mathcal{B}}{\leq} 0$.
		\item $\B$ is asymptotically stable if and only if there exists $\Psi \in \R_s^{q \times q}[\zeta, \eta]$ such that $Q_\Psi \stackrel{\mathcal{B}}{\geq} 0$ and $\frac{d}{dt}Q_\Psi \stackrel{\mathcal{B}}{<} 0$.
	\end{enumerate}
	\label{t: lyap stab}
\end{theorem}

\subsection{Dissipativity}
In what follows, we define and discuss dissipativity in the behavioral context. Let $\B \in \mathcal{L}^q$, $\Phi \in \R_s^{q \times q}[\zeta,\eta]$, and consider the associated QDF $Q_\Phi$, which we call the \emph{supply rate}. Let $N$ be the order of $Q_\Phi$. If $N=0$, we call the supply rate $Q_\Phi$ \emph{static}; otherwise, we call it \emph{dynamic}.
\begin{definition}
	We call $\B$ \emph{dissipative} with respect to $Q_\Phi$ if
	\begin{equation*}
		\int_{-\infty}^\infty Q_\Phi(w)dt \geq 0 
	\end{equation*}
	for all $w\in \B \cap \mathcal{D}(\R,\R^q)$. We call $\B$ \emph{half-line dissipative} with respect to $Q_\Phi$ if
	\vspace{-0.5ex}
	\begin{equation*}
		\int_{-\infty}^0 Q_\Phi(w)dt \geq 0 
	\end{equation*}
	for all $w\in \B \cap \mathcal{D}(\R,\R^q)$. Moreover, if these inequalities hold strictly for all nonzero $w \in \B \cap \D(\R,\R^q)$, we call $\B$ \emph{strictly dissipative} with respect to $Q_\Phi$ and \emph{strictly half-line dissipative} with respect to $Q_\Phi$, respectively. 
\end{definition}

Now we define related concepts of storage functions and dissipation inequalities.
\begin{definition}
	Let $\B \in \mathcal{L}^q$ and $\Phi \in \R_s^{q \times q}[\zeta,\eta]$. We call the QDF $Q_\Psi$ induced by $\Psi \in \R_s^{q \times q}[\zeta,\eta]$ a \emph{storage function for $(\B, Q_\Phi)$} if it satisfies the \emph{dissipation inequality}
	\vspace{-0.5ex}
	\begin{equation}
		\frac{d}{dt}Q_\Psi \stackrel{\mathcal{B}}{\leq} Q_\Phi.
		\label{diss ineq}
	\end{equation}
	Moreover, if \eqref{diss ineq} holds strictly, we say $Q_{\Psi}$ satisfies the dissipation inequality strictly.
\end{definition}
\begin{remark}
Note that the dissipation inequality can be seen as finding a QDF $Q_\Psi$ such that the QDF $Q_\Phi -\frac{d}{dt}Q_\Psi$ is nonnegative on $\B$. This can be verified by solving an LMI for a coefficient matrix of $Q_\Psi$; for example, see \cite[Algorithm~8]{belur_algorithmic_2002}. 
\end{remark}

For controllable systems, dissipativity is equivalent to the existence of storage functions. 
\begin{theorem}[\hspace{-0.01ex}{\cite{willems_quadratic_1998, trentelman_every_1997}}]
	Let $\B \in \mathcal{L}^q$ be controllable and consider ${\Phi\in \R_s^{q \times q}[\zeta, \eta]}$. 
	The following statements hold.
	\begin{enumerate}[label=\roman*)]
		\item $\B$ is dissipative with respect to $Q_\Phi$ if and only if there exists a storage function $Q_\Psi$ for $(\B, Q_\Phi)$. \label{first stat of thm diss char}
		\item  $\B$ is half-line dissipative with respect to $Q_\Phi$ if and only if there exists a storage function $Q_\Psi$ for $(\B, Q_\Phi)$ such that $Q_\Psi \stackrel{\mathcal{B}}{\geq }0$. \label{second stat of thm diss char}
	\end{enumerate}
	\label{t: diss char}
\end{theorem}
Statement~\ref{first stat of thm diss char} was shown in \cite[Theorem~4.3]{trentelman_every_1997}. Moreover, \ref{second stat of thm diss char} follows from \cite[Theorem~6.3]{willems_quadratic_1998} applied to ${\hat{\Phi}(\zeta,\eta):= M^\top (\zeta)\Phi(\zeta,\eta)M(\eta)}$, where $M \in \R^{q \times r}[\xi]$ is such that $\B=\im M\left(\frac{d}{dt}\right)$.

\begin{remark}
	Unlike in the original definition of dissipativity in \cite{willems_dissipative_1972}, we do not define dissipativity in terms of the existence of nonnegative storage functions. This is in line with other dissipativity studies within the behavioral approach  \cite{trentelman_every_1997, willems_quadratic_1998}. We note that dissipativity with a possibly indefinite storage function is sometimes referred to as \emph{cyclo-dissipativity}, see e.g. \cite[Section~3.1]{van_der_schaft_dissipative_2017}.
\end{remark}

Let $\B \in \mathcal{L}^q$ and $\Phi \in \mathbb{R}_s^{q \times q}[\zeta,\eta]$. Suppose that $Q_\Psi$ is a storage function for $(\B,Q_\Phi)$. For various choices of $\Phi$, the dissipation inequality \eqref{diss ineq} captures many dissipativity properties in the literature. Namely, we say that $\B$ has \linebreak \emph{$L_2$-gain less than or equal to $\gamma$} if \eqref{diss ineq} holds with \begin{equation*}\Phi(\zeta,\eta)=\begin{bmatrix}
			\gamma^2 I & 0 \\0 & -I
		\end{bmatrix}, \quad \gamma>0 	\vspace{-1ex}\end{equation*} 
	and $Q_{\Psi}\stackrel{\B}{\geq}0$, $\B$ is \emph{passive} if \eqref{diss ineq} holds with
\begin{equation}\Phi(\zeta,\eta)=\frac{1}{2}\begin{bmatrix}
			0 & I \\ I & 0
		\end{bmatrix} \label{passive sr} \vspace{-1ex}\end{equation} and $Q_\Psi \stackrel{\B}{\geq}0$, and $\B$ is \emph{delta dissipative} with supply rate $Q_\Phi$ if \eqref{diss ineq} holds with $\Phi(\zeta,\eta)=\zeta \eta X$ with $X \in \R^{q \times q}$ and $Q_\Psi \stackrel{\B}{\geq}0$. Lastly, $\B$ is \emph{negative imaginary} if \eqref{diss ineq} holds with \begin{equation}\Phi(\zeta,\eta)=\frac{1}{2}\begin{bmatrix}
			0 & \eta I \\
			\zeta I & 0
		\end{bmatrix}. \label{NI sr} 
\end{equation}
\section{Problem statement}
\label{sec:problem statement}
In a behavioral setting, we do not a priori distinguish between inputs and outputs and instead view individual sub-systems in terms of the behavior of their shared variables. Hence, an interconnection consists of requiring that these variables satisfy the laws of all sub-systems. For more details and examples of behavioral interconnections, we refer the reader to \cite{willems_interconnections_1997}.

We focus on interconnections of two LTI systems that are dissipative with respect to some (dynamic) supply rates and aim to find conditions under which this interconnection is stable. 

Consider $\Sigma_1=(\R,\R^q, \mathcal{B}_1)$ and ${\Sigma_2=(\R,\R^q, \mathcal{B}_2)}$, and let $\Phi_1,\Phi_2\in \R^{q \times q}_s[\zeta,\eta]$ induce the supply rates $Q_{\Phi_1},Q_{\Phi_2}$. Suppose that $Q_{\Psi_i}$ is a storage function for $(\B_i,Q_{\Phi_i})$, for $i=1,2$.

We define an interconnection of $\Sigma_1$ and $\Sigma_2$, depicted in Fig. \ref{clc inter},  as the system ${\Sigma_1 \land \Sigma_2:= (\R,\R^q, \B)}$, where ${\B:=\B_1 \cap \B_2}$. That is, a system whose trajectories are compatible with the laws of both $\Sigma_1$ and $\Sigma_2$.  
\vspace{-2ex}
\begin{figure}[thpb]
	\centering
	\begin{tikzpicture}[auto,node distance = 4mm and 22mm]
		\node  [block, minimum width=16mm, minimum height=10mm] (system1) {$\Sigma_1$};
		\node  [block,  minimum width=16mm, minimum height=10mm, right = of system1] (system2) {$\Sigma_2$};
		\coordinate[above right= of system1.east] (w);
		\coordinate[right= of system1.east] (w2);
		\coordinate[below right= of system1.east] (w3);
		\draw([yshift=8pt]system1.east) -- ++(2.21,0);
		\draw ([yshift=10pt]system1.east) -- ++(2.21,0);
		\node at ([yshift=7.5pt, xshift=-31.5pt]w) {$w$};
		\draw([yshift=-8pt]system1.east) -- ++(2.21,0);
		\draw ([yshift=-10pt]system1.east) -- ++(2.21,0);
		\draw ([yshift=0pt]system1.east)  (1.91,0.1)node {$\vdots$}(w2);
	\end{tikzpicture} 
	\caption{Full interconnection of $\Sigma_1$ and $\Sigma_2$.}
	\label{clc inter}
\end{figure}
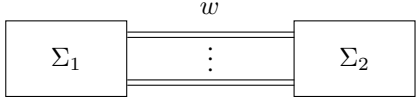
\vspace{-1ex}
\begin{problem}
	Find conditions on the behaviors $\B_1,\B_2$, the supply rates $Q_{\Phi_1}$, $Q_{\Phi_2}$, and the storage functions $Q_{\Psi_1}$, $Q_{\Psi_2}$ such that the interconnected behavior $\B$ is (asymptotically) stable.
\end{problem}

The interconnected behavior $\B$ can capture the trajectories of many well-known (feedback) interconnections. In the following, we describe a few.
\begin{example}
	Consider  $\bar{\B}_1\in \mathcal{L}^{m_1+p_1}$ and $\bar{\B}_2\in \mathcal{L}^{m_2+p_2}$. For $i=1,2$, partition ${w_i\in \bar{\B}_i}$ as ${w_i=(u_i,y_i)}$, where ${u_i\in \left(\R^{m_i}\right)^\R}$ and ${y_i\in \left(\R^{p_i}\right)^\R}$. Consider the interconnection obtained by setting
	\vspace{-1ex} \begin{equation}\begin{bmatrix}
			u_1 \\ u_2
		\end{bmatrix}=K\begin{bmatrix}
			y_1 \\ y_2
		\end{bmatrix},\label{K law}\end{equation} where ${K\in \R^{(m_1+m_2)\times(p_1+p_2)}}$. We call such an interconnection a \emph{$K$-interconnection}.
	Let $q:=p_1+p_2$. A natural choice for a behavior of the interconnection is the following:
	\[\B := \set{(y_1,y_2)\in \left(\R^{q}\right)^\R}{\begin{gathered}\exists (u_1,u_2) 
			\text{ s.t. } (u_1,y_1) \in \bar{\B}_1,\\ (u_2,y_2) \in \bar{\B}_2, \text{ and \eqref{K law} holds }\end{gathered}}.\]
	This behavior is equal to the intersection of two auxiliary behaviors. Indeed, if we define
	\begin{align}
		K_1&:=\begin{bmatrix}
			I_{m_1} & 0_{m_1 \times m_2} & 0 & 0 \\ 0 & 0 & I_{p_1} & 0_{p_1\times p_2}
		\end{bmatrix}\begin{bmatrix}
			K \\ I
		\end{bmatrix}, \label{K1} \\ K_2&:=\begin{bmatrix}
			0_{m_2 \times m_1} & I_{m_2} & 0 & 0 \\ 0 & 0 & 0_{p_2 \times p_1} & I_{p_2}
		\end{bmatrix}\begin{bmatrix}
			K \\ I
		\end{bmatrix},\label{K2}
	\end{align}
	\vspace{-2ex}\\
and
	\begin{equation}
		\B_1:=\set{w}{K_1w\in \bar{\B}_1}, \;\;
		\B_2:=\set{w}{K_2w\in \bar{\B}_2},
		\label{arbitrary sr barB}
	\end{equation} then $\B=\B_1\cap\B_2$.
	
	We note that $K$-interconnections include positive and negative feedback
	interconnections, depicted in Fig. \ref{fig: feedback}. Concretely, let $m_1=p_1$, $m_2=p_2$, and
	\begin{equation}
		K_+ = \begin{bmatrix}
			0 & I_{p_2} \\ I_{p_1} & 0
		\end{bmatrix}
		\quad \text{and} \quad
		K_- = \begin{bmatrix}
			0 & -I_{p_2} \\ I_{p_1} & 0
		\end{bmatrix}.
		\label{feedback Ks}
	\end{equation}
	For $K=K_+$,
	$\B$ consists precisely of all trajectories $(y_1,y_2)$ such that $(y_2,y_1) \in \bar{\B}_1$ and $(y_1,y_2) \in \bar{\B}_2$. As such, $\B$ captures the positive feedback interconnection of $\bar{\B}_1$ and $\bar{\B}_2$. 
	Moreover, for $K=K_-$, 
	$\B$ consists of all trajectories $(y_1,y_2)$ such that $(-y_2,y_1) \in \bar{\B}_1$ and $(y_1,y_2) \in \bar{\B}_2$. In this case, $\B$ captures the negative feedback interconnection of $\bar{\B}_1$ and $\bar{\B}_2$.
	\vspace{-3ex}
	\begin{figure}[H]
		\centering
		\begin{subfigure}[b]{0.5\linewidth}
			\centering
			\begin{tikzpicture}[auto,node distance = 4mm and 10mm,baseline=(current bounding box.center)]
				\node  [block, minimum width=10mm, minimum height=6mm] (system1) {$\Sigma_1$};
				\node  [block, minimum width=10mm, minimum height=6mm, below = of system1] (system2) {$\Sigma_2$};
				\coordinate[left = of system1.west] (u1);
				\coordinate[right= of system1.east] (y1);
				\coordinate[left = of system2.west] (u2);
				\coordinate[right= of system2.east] (y2);
				
				\draw[-stealth] (u1) -- node[above] {$u_1$} (system1);
				\draw (system1) --node[above] {$y_1$} (y1);
				\draw[-stealth] (y1) |- node[near end, above]{$u_2$}(system2);
				\draw (system2) -|  node[near start, above]{$y_2$}(u1);
			\end{tikzpicture}
			\caption{Positive feedback.}
			\label{fig: pos feedback}
		\end{subfigure}%
		\begin{subfigure}[b]{0.5\linewidth}
			\centering
			\begin{tikzpicture}[auto,node distance = 4mm and 10mm,baseline=(current bounding box.center)]
				\node  [block, minimum width=10mm, minimum height=6mm] (system1) {$\Sigma_1$};
				\node  [block, minimum width=10mm, minimum height=6mm, below = of system1] (system2) {$\Sigma_2$};
				\node[sum, left = of system1.west] (sum) {};
				\coordinate[left = of system1.west] (u1);
				\coordinate[right= of system1.east] (y1);
				\coordinate[left = of system2.west] (u2);
				\coordinate[right= of system2.east] (y2);
				
				\draw[-stealth] (sum) -- node[above] {$u_1$} (system1);
				\draw (system1) --node[above] {$y_1$} (y1);
				\draw[-stealth] (y1) |- node[near end, above]{$u_2$}(system2);
				\draw [-stealth](system2) -|  node[near start, above]{$y_2$}(sum);
				
				\node[below left= 0mm and 2mm of sum.east] {$-$};
			\end{tikzpicture}
			\caption{Negative feedback.}
			\label{fig: neg feedback}
		\end{subfigure}
		\caption{Feedback interconnections.}
		\label{fig: feedback}
	\end{figure}
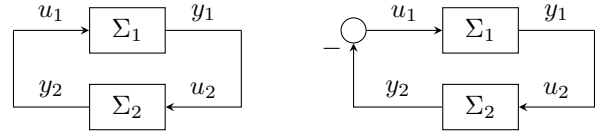
	\vspace{-3ex}
	\label{ex K-inter}
\end{example}

\section{A basic behavioral interconnection theorem} 
\label{sec:first theorem}
The following basic theorem provides conditions under which the interconnection of two dissipative behaviors is (asymptotically) stable.
\begin{theorem}
	Let $\B_1,\B_2 \in \mathcal{L}^q$ and ${\Phi_1, \Phi_2 \in \R_s^{q \times q}[\zeta,\eta]}$. Assume that exists a storage function $Q_{\Psi_i}$ for $(\B_i, Q_{\Phi_i})$ such that $Q_{\Psi_i}\stackrel{\B_i}{\geq}0$, for $i=1,2$. Define $\B := \B_1 \cap \B_2$ and assume that there exist scalars $\alpha_1>0$ and $\alpha_2 \geq0$ such that
	\vspace{-0.5ex}
	\begin{equation}
		\alpha_1 Q_{\Phi_1}+\alpha_2 Q_{\Phi_2} \stackrel{\B}{\leq} 0.
		\label{diss inter thm cond}
		\vspace{-1ex}
	\end{equation}Then,
	\begin{enumerate}[label=\roman*)]
		\item $\B$ is stable if $Q_{\Psi_1}\stackrel{\B_1}{>}0$;  \label{first stat first diss inter thm}
		\item $\B$ is asymptotically stable if $Q_{\Psi_1}$ satisfies the dissipation inequality strictly. \label{second stat first diss inter thm}
	\end{enumerate}
	\label{t: first diss inter thm}
\end{theorem}
We omit the proof as this theorem is a special case of our more general result, Theorem~\ref{t: main diss inter thm}, as discussed in Example~\ref{ex. prev thm}. Moreover, we note that the conditions of this theorem can be verified using LMIs; see Remark~\ref{main thm remark} for more details.

\subsection{Examples}
Theorem~\ref{t: first diss inter thm} generalizes many stability interconnection results in the literature. In this section, we discuss how application of our theorem to specific $K$-interconnections can recover various known results.

Consider $\bar{B}_1\in \mathcal{L}^{m_1+p_1}, \;\bar{\B}_2\in \mathcal{L}^{m_2+p_2}$ and for ${K\in \R^{(m_1+m_2)\times(p_1+p_2)}}$ their $K$-interconnection \eqref{K law} with the behavior ${\B=\B_1\cap \B_2}$, where $\B_1, \B_2$ are as in \eqref{arbitrary sr barB}. 

Partition $K$ such that
\begin{equation}
	K=\begin{bmatrix}
		K_{11} & K_{12}\\ K_{21} & K_{22}
	\end{bmatrix},
	\label{K partition}
\end{equation}
where $K_{11}\in \R^{m_1\times p_1}$,
$K_{12}\in \R^{m_1\times p_2}$, $K_{21}\in \R^{m_2\times p_1}$, and $K_{22}\in \R^{m_2\times p_2}$. 

\begin{example}[Arbitrary static supply rates]
\label{ex: arbitrary sr}
Assume that $K_{12}$ has full column rank.
For $i=1,2$, let \[{\bar{\Phi}_i(\zeta,\eta)=X^i=\begin{bmatrix}
		X_{11}^i & X_{12}^i \\ X_{21}^i & X_{22}^i
\end{bmatrix}},\]
where $X_{11}^i\in \R^{m_i \times m_i}$, $X_{12}^i\in \R^{m_i \times p_i}$, $X_{21}^i\in \R^{p_i \times m_i}$, and $X_{22}^i\in \R^{p_i \times p_i}$. Suppose that there exists a storage function $Q_{\bar{\Psi}_i}$ for $(\bar{\B}_i,Q_{\bar{\Phi}_i})$, for $i=1,2$, such that $Q_{\bar{\Psi}_1}\stackrel{\bar{\B}_1}{>}0$ and  $Q_{\bar{\Psi}_2}\stackrel{\bar{\B}_2}{\geq}0$. Moreover, in line with \cite{arcak_stability_2016}, assume that there exist $\alpha_1>0$ and $\alpha_2\geq 0$ such that $\Xi(\alpha_1,\alpha_2)\leq0$, where $\Xi(\alpha_1,\alpha_2)$ is defined as
	\begin{equation}
		\begin{bmatrix}
			K \\ I
		\end{bmatrix}^\top \begin{bmatrix}
			\alpha_1X^{11}_1 & 0  &\alpha_1X_1^{12} & 0\\
			0 & \alpha_2 X_2^{11} & 0 & \alpha_2X_2^{12}\\
			\alpha_1X^{21}_1 & 0  &\alpha_1X_1^{22} & 0\\
			0 & \alpha_2 X_2^{21} & 0 & \alpha_2X_2^{22}
		\end{bmatrix}\begin{bmatrix}
			K \\ I
		\end{bmatrix}.
		\label{Xi def}
	\end{equation}
	For $i=1,2$, define
	\begin{equation}
		\Phi_i(\zeta,\eta):=K_i^\top \bar{\Phi}_i(\zeta,\eta)K_i, \;\; \Psi_i(\zeta,\eta):=K_i^\top \bar{\Psi}(\zeta,\eta)K_i,
		\label{bar Phi and bar Psi}
	\end{equation}
	where $K_1, K_2$ are as in \eqref{K1} and \eqref{K2}, respectively.
	Then, by Lemma~\ref{lemma barB dissipativity} parts \ref{lemma stat 1} and \ref{lemma stat 3}, $Q_{\Psi_i}$ is a storage function for $(\B_i,Q_{\Phi_i})$ such that $Q_{\Psi_i}\stackrel{\B_i}{\geq}0$. Moreover, as $K_{12}$ has full column rank, also $K_1$
	has full column rank. Then, by Lemma~\ref{lemma barB dissipativity} part \ref{lemma stat 4}, $Q_{\Psi_1}\stackrel{\B_1}{>}0$.  Moreover, for $\Xi$ as in \eqref{Xi def} and all $w\in \left(\R^q\right)^\R$, we have that 
	\begin{equation*}
		\alpha_1Q_{\Phi_1}(w)+\alpha_2Q_{\Phi_2}(w)= w^\top \Xi(\alpha_1,\alpha_2) w\leq 0.
	\end{equation*}
	Hence, by Theorem~\ref{t: first diss inter thm}, the interconnected behavior $\B$ is stable.
	
	Therefore, Theorem~\ref{t: first diss inter thm} allows us to show a similar result to the generalized interconnection result for networks of dissipative systems with static supply rates in \cite[Proposition~2.1]{arcak_stability_2016}, adapted to our behavioral setting and applied to an interconnection of two LTI systems. 
	
	Unlike in \cite[Proposition~2.1]{arcak_stability_2016}, we do not require that the interconnection is well-posed or that both of the storage functions are positive. However, we add an extra rank assumption on the $K$ matrix. We note that the full column rank assumption on $K_{12}$ is not too restrictive, as common examples of $K$-interconnections like positive/negative feedback (with $K=K_+$ or $K_-$ in \eqref{feedback Ks}) satisfy it.
	
	Lastly, we note that with this extra assumption, one can similarly show that, if in addition, $Q_{\Psi_1}$ satisfies the dissipation inequality strictly, the interconnected behavior $\B$ will be asymptotically stable.
	
	\label{ex arbitrary static sr}
\end{example}
\begin{example}[Small gain theorem]
	Assume that $K=K_+$ as in \eqref{feedback Ks}.
	Let $\gamma_1,\gamma_2>0$ be such that $\gamma_1 \gamma_2 <1$ and, for $i=1,2$, assume that $\bar{\B}_i$ has $L_2$-gain less than or equal to $\gamma_i$. Define \vspace{-0.5ex}
	\[\bar{\Phi}_i(\zeta,\eta):=\begin{bmatrix}
		\gamma_i^2 I & 0 \\ 0 & -I
	\end{bmatrix},
	\vspace{-1ex}\]
	for $i=1,2$.
	Let $Q_{\bar{\Psi}_i}\stackrel{\bar{\B}_i}{\geq}0$ be a storage function associated with $(\bar{\B}_i,Q_{\bar{\Phi}_i})$. 
	
	For $i=1,2$ define $\Phi_i,\Psi_i$ as in \eqref{bar Phi and bar Psi} with $K_i$ as in \eqref{K1},\eqref{K2}.
	Then, by Lemma~\ref{lemma barB dissipativity} parts \ref{lemma stat 1} and \ref{lemma stat 3}, $Q_{\Psi_i}$ is a storage function for $(\B_i, Q_{\Phi_i})$ such that ${Q_{\Psi_i} \stackrel{\B_i}{\geq}0}$. 
	
	Let $\hat{\gamma}_1>\gamma_1$ be such that $\hat{\gamma}_1\gamma_2<1$. Then there exists $\delta<1$ such that $\hat{\gamma}_1\gamma_2\leq \delta$. Let \[\hat{\Phi}_1(\zeta,\eta):=\begin{bmatrix}
		-\delta^2 I & 0 \\0 & \hat{\gamma}_1^2 I\\
	\end{bmatrix}.\]
	We have that
	\[\frac{d}{dt}Q_{\Psi_1} \stackrel{\B_1}{\leq}
	Q_{\Phi_1}< Q_{\hat{\Phi}_1},\]
	as for all $w=(u,y)\in \left(\R^q\right)^\R $,
	\[Q_{\Phi_1}(w)=\gamma_1^2 \|y\|^2 - \|u\|^2\leq \hat{\gamma_1}^2\|y\|^2-\delta^2 \|u\|^2 = Q_{\Phi_1}(w),\]
	and $Q_{\Phi_1}(w) = Q_{\hat{\Phi}_1}(w)$ if and only if $w=0$.
	Hence, $Q_{\Psi_1}$ is a storage function for $(\B_1,Q_{\hat{\Phi}_1})$ such that it satisfies the dissipation inequality strictly. Then, for ${\alpha_1 =1}$, ${\alpha_2 = \hat{\gamma}_1^2}$, and all $w\in \left(\R^q\right)^\R$, we get that
	\begin{gather*}
		\alpha_1 Q_{\hat{\Phi}_1}(w)+\alpha_2Q_{\Phi_2}(w) = 
		w^\top \begin{bmatrix}
			\left( \hat{\gamma}_1^2\gamma_2^2-\delta^2\right) I & 0 \\ 0 & 0
		\end{bmatrix}w \leq 0.
	\end{gather*} 
	Hence, by Theorem~\ref{t: first diss inter thm}, the interconnected behavior $\B$ is asymptotically stable. This recovers a version of the small gain theorem; see, for example, \cite{zhou_essentials_1998}.
\end{example}
\begin{example}[Passivity theorem]
	Assume that $K=K_-$ as in \eqref{feedback Ks} and suppose that $\bar{\B}_1$ and $\bar{\B}_2$ are passive. Let $\bar{\Phi}_i$ be as in \eqref{passive sr}
	for $i=1,2$, and let $Q_{\bar{\Psi}_i}\stackrel{\bar{\B}_i}{\geq}0$ be the associated storage function for $(\bar{\B}_i,Q_{\bar{\Phi}_i})$. Assume that $Q_{\bar{\Psi}_1}\stackrel{\bar{\B}_1}{>}0$.
	
	Notice that $K_1$ as in \eqref{K1} has full column rank. Moreover, for $\alpha_1=\alpha_2=1$, $\Xi(\alpha_1,\alpha_2)$ as in \eqref{Xi def} satisfies ${\Xi(\alpha_1,\alpha_2)=0}$. 
	Hence, this is a special case of Example \ref{ex: arbitrary sr} and, as we have shown, by Theorem~\ref{t: first diss inter thm}, $\B$ is stable. 
	
	This recovers a version of the passivity theorem; see, for example, \cite{desoer_feedback_2009}.
	
\end{example}

\begin{example}[Delta dissipative systems]			
	Assume that $K_{12}$ is column rank and that $\bar{\B}_1$ and $\bar{\B}_2$ are delta dissipative with the supply rates $Q_{\bar{\Phi}_i}$ such that
	\[\bar{\Phi}_i(\zeta,\eta)=\zeta \eta X^i=\zeta \eta \begin{bmatrix}
		X_{11}^i & X_{12}^i \\ X_{21}^i & X_{22}^i
	\end{bmatrix},\]where $X_{11}^i\in \R^{m_i \times m_i}$, $X_{12}^i\in \R^{m_i \times p_i}$, $X_{21}^i\in \R^{p_i \times m_i}$, and $X_{22}^i\in \R^{p_i \times p_i}$, for $i=1,2$.
	Moreover, assume that there exist $\alpha_1>0$ and $\alpha_2\geq 0$ such that $\Xi(\alpha_1,\alpha_2)$, as in \eqref{Xi def}, satisfies ${\Xi(\alpha_1,\alpha_2)\leq0}$.
	
	For $i=1,2$, define $\Phi_i,\Psi_i$ as in \eqref{bar Phi and bar Psi}. Then, using that $K_1$ as in \eqref{K1} has full column rank and by Lemma~\ref{lemma barB dissipativity} parts \ref{lemma stat 1}, \ref{lemma stat 3}, and \ref{lemma stat 4}, $Q_{\Psi_i}$ is a storage function for $(\B_i,Q_{\Phi_i})$ for $i=1,2$, such that $Q_{\Psi_1} \stackrel{\B_1}{>}0$ and $Q_{\Psi_2} \stackrel{\B_2}{\geq}0$. 
	Then, for all $w\in \left(\R^q\right)^\R$ we have that
	\begin{equation*}		\alpha_1Q_{\Phi_1}(w)+\alpha_2Q_{\Phi_2}(w)=\dot{w}^\top \Xi(\alpha_1,\alpha_2) \dot{w}\leq 0. 
	\end{equation*}
	Hence, by Theorem~\ref{t: first diss inter thm}, the interconnected behavior $\B$ is stable. This shows a similar interconnection result for networks of delta dissipative systems to \cite[Corollary~1]{schweidel_compositional_2022}, adapted to our behavioral setting and applied to an interconnection of two systems. 
	
	Just as in Example~\ref{ex arbitrary static sr}, we note that we have slightly different assumptions than in \cite[Corollary~1]{schweidel_compositional_2022}. However, as was previously explained, they are not too restrictive.
	
	Moreover, if in addition, $Q_{\Psi_1}$ satisfies its dissipation inequality strictly, we can use Theorem~\ref{t: first diss inter thm} to conclude that $\B$ is asymptotically stable.
\end{example}
\begin{remark}
	\label{r: another behavioral}
	In \cite{takaba_stability_2005}, the stability of an interconnection of two controllable, dissipative behaviors was studied. Our theorem is also applicable to this case.
	
	Let $\B_1,\B_2 \in \mathcal{L}^q$ be controllable and consider ${\Phi \in \R_s^{q \times q}[\zeta,\eta]}$. Suppose that there exists $\varepsilon>0$ such that
	\begin{equation}
		\int_{-\infty}^0 \left(Q_{\Phi}(w)\right)(\tau) d\tau\geq \epsilon \int_{-\infty}^0 \|w(\tau)\|^2 d\tau,
		\label{other def strict diss}
	\end{equation}
	for all $w \in \B_1 \cap \mathcal{D}(\R, \R^q)$, and that $\B_2$ is half-line dissipative with respect to $Q_{-\Phi}$. In \cite[Theorem~1]{takaba_stability_2005}, the authors showed that under these assumptions, the behavior $\B=\B_1 \cap \B_2$ is \linebreak$L_2$-stable. That is, for all $w\in \B$ we have that
	\begin{equation*}
		\int_0^\infty \|w(t)\|^2 dt <+\infty.
	\end{equation*}
	Let $\Phi_1=\Phi$ and $\Phi_2=-\Phi$. By Theorem~\ref{t: diss char}, there exists a storage function $Q_{\Psi_i}$ for $(\B_i,Q_{\Phi_i})$ such that ${Q_{\Psi_i}\stackrel{\B_i}{\geq}0}$, for $i=1,2$. Moreover, \eqref{other def strict diss} implies that $\B_1$ is strictly half-line dissipative with respect to $Q_{\Phi_1}$ and that $Q_{\Psi_1}$ satisfies the dissipation inequality strictly.
	Thus, by Theorem~\ref{t: first diss inter thm}, $\B$ is asymptotically stable.
\end{remark}

\section{A unifying behavioral interconnection theorem}
\label{sec:main theorem}
We have discussed several examples of interconnection results in the literature that Theorem~\ref{t: first diss inter thm} generalizes. However, in some cases, Theorem~\ref{t: first diss inter thm} is not applicable. One such example is the case of negative imaginary systems \cite{xiong_negative_2010}.

\begin{example}[Negative imaginary systems]
	Consider negative imaginary behaviors $\bar{\B}_1, \bar{\B}_2\in \mathcal{L}^q$. Let $\bar{\Phi}_i$ be as in \eqref{NI sr}
	for $i=1,2$ be the associated supply rate for $\bar{\B}_i$. Consider the positive feedback interconnection of $\bar{\B}_1$ and $\bar{\B}_2$ with the behavior $\B=\B_1\cap \B_2$, where $\B_1, \B_2$ are as in \eqref{arbitrary sr barB} with $K=K_+$ as in \eqref{feedback Ks}. 
	Define $\Phi_i,\Psi_i$ as in \eqref{bar Phi and bar Psi}, for $i=1,2$. Then, by Lemma~\ref{lemma barB dissipativity} part \ref{lemma stat 1}, $Q_{\Psi_i}$ is a storage function for $(\B_i,Q_{\Phi_i})$, for $i=1,2$. Moreover, for all $\alpha_1, \alpha_2\geq 0$
	\begin{gather*}
		\alpha_1\Phi_1(\zeta,\eta)\!+\!\alpha_2\Phi_2(\zeta,\eta)\!=\!\frac{1}{2}\!\begin{bmatrix}
			0 & \hspace{-2ex}(\alpha_2\zeta +\alpha_1\eta) I\\
			\hspace{-0.25ex}(\alpha_1\zeta+\alpha_2\eta)I & 0
		\end{bmatrix},
	\end{gather*}
	and hence, $Q_\Phi:=\alpha_1Q_{\Phi_1}+\alpha_2Q_{\Phi_2}$ is the QDF with the coefficient matrix
	\begin{equation*}
		\tilde{\Phi}=\frac{1}{2}\begin{bmatrix}
			0 & 0 & 0 & \alpha_1 I \\
			0 & 0 & \alpha_2 I & 0\\
			0 & \alpha_2 I & 0 & 0\\
			\alpha_1 I & 0 & 0 & 0
		\end{bmatrix}.
	\end{equation*}
	We see that this coefficient matrix is indefinite for any choice of $\alpha_1>0,\alpha_2\geq 0$ and hence we cannot, in general, conclude that ${Q_\Phi \stackrel{\B}{\leq }0}$. This means that the condition \eqref{diss inter thm cond} is not necessarily satisfied. Hence, Theorem~\ref{t: first diss inter thm} is not applicable for interconnections of negative imaginary systems.
	
	Nonetheless, in \cite{xiong_negative_2010}, the authors have shown that, under a DC gain condition, a positive feedback interconnection of negative imaginary systems is asymptotically stable. For example, consider the behaviors $\bar{\B}_1,\bar{\B}_2 \in \mathcal{L}^{2}$ with the kernel representations $R_i\left(\frac{d}{dt}\right)w=0$ where
	\begin{equation}
		\begin{aligned}
			R_1(\xi)=& 
			\begin{bmatrix}
				-1 & \xi+1
			\end{bmatrix}, \\
			R_2(\xi)=&
			\begin{bmatrix}
				-2\xi^2-\xi-1 & 2\xi^4+7\xi^3+17\xi^2+17\xi+5
			\end{bmatrix},
		\end{aligned}
		\label{Ris of NI systems}
	\end{equation}
	respectively. In \cite{petersen_negative_2016}, the positive feedback interconnection of $\bar{\B}_1$ and $\bar{\B}_2$ was shown to be asymptotically stable. 
	
	Therefore, even though Theorem~\ref{t: first diss inter thm} is not applicable to interconnections of negative imaginary systems, such interconnections can be stable. This example suggests that one needs to consider a more general condition on the supply rates of dissipative systems than \eqref{diss inter thm cond}.
	\label{ex: NI}
\end{example}

The following theorem is the main result of this paper.
\begin{theorem}
	Let $\B_1,\B_2 \in \mathcal{L}^q$ and ${\Phi_1, \Phi_2 \in \R_s^{q \times q}[\zeta,\eta]}$. 
	Suppose there exists a storage function $Q_{\Psi_i}$ for $(\B_i, Q_{\Phi_i})$, for $i=1,2$. Define $\B := \B_1 \cap \B_2$ and assume that there exist scalars $\alpha_1>0$ and $\alpha_2\geq 0$, and $\Gamma\in \R_s^{q \times q}[\zeta,\eta]$ such that 
	\begin{equation}
		\frac{d}{dt}Q_\Gamma \stackrel{\B}{\geq}\alpha_1Q_{\Phi_1}+\alpha_2Q_{\Phi_2}.
		\label{main thm cond 1}
		\vspace{-1ex}
	\end{equation}
	Then,
	\begin{enumerate}[label=\roman*)]
		\item $\B$ is stable if
		\vspace{-1.5ex}
		\begin{equation}
			\alpha_1Q_{\Psi_1}+\alpha_2Q_{\Psi_2}-Q_\Gamma \stackrel{\B}{>} 0;
			\label{main thm cond strict}
		\end{equation} \label{main thm stat 1}
		\vspace{-2ex}
		\item $\B$ is asymptotically stable if $Q_{\Psi_1}$ satisfies the dissipation inequality strictly and
		\vspace{-1.5ex}
		\begin{equation}
			\alpha_1Q_{\Psi_1}+\alpha_2Q_{\Psi_2}-Q_\Gamma \stackrel{\B}{\geq} 0.
			\label{main thm cond non strict}
		\end{equation}\label{main thm stat 2}
	\end{enumerate} 
	\label{t: main diss inter thm}
\end{theorem}
\vspace{-3ex}
Note that Theorem~\ref{t: main diss inter thm} replaces the non-positivity condition \eqref{diss inter thm cond} by the more general condition \eqref{main thm cond 1}. This introduces another QDF $Q_\Gamma$ that, in general, depends on both behaviors $\B_1$ and $\B_2$. For this reason, $Q_\Gamma$ is referred to as the \emph{coupling QDF}. An interesting aspect of Theorem~\ref{t: main diss inter thm} is that the Lyapunov functions in \eqref{main thm cond strict} and \eqref{main thm cond non strict} are not mere sums of storage functions since they also depend on the coupling QDF.

\begin{proof}[Proof of Theorem~\ref{t: main diss inter thm}]
	First, we show \ref{main thm stat 1}. Consider the candidate Lyapunov function
	\begin{equation}
		Q_\Psi:= \alpha_1 Q_{\Psi_1}+\alpha_2 Q_{\Psi_2}-Q_\Gamma.
		\label{main thm lyap}
	\end{equation}
	By assumption, we have that $Q_\Psi \stackrel{\B}{>}0$. Moreover, 
	\begin{equation*}
		\begin{aligned}
			\frac{d}{dt}Q_\Psi &= \alpha_1\frac{d}{dt}Q_{\Psi_1}+\alpha_2 \frac{d}{dt}Q_{\Psi_2}  -\frac{d}{dt}Q_\Gamma \\
			&\stackrel{\B}{\leq} \alpha_1 Q_{\Phi_1} + \alpha_2 Q_{\Phi_2}-\frac{d}{dt}Q_\Gamma \stackrel{\B}{\leq} 0.
		\end{aligned}
	\end{equation*}
	Thus, by Theorem~\ref{t: lyap stab}, $\B$ is stable.
	
	Now we show \ref{main thm stat 2}. By assumption, the candidate Lyapunov function \eqref{main thm lyap} satisfies $Q_\Psi \stackrel{\B}{\geq }0$. Moreover, 
	\begin{equation*}
		\begin{aligned}
			\frac{d}{dt}Q_\Psi &= \alpha_1\frac{d}{dt}Q_{\Psi_1}+\alpha_2 \frac{d}{dt}Q_{\Psi_2}  -\frac{d}{dt}Q_\Gamma \\
			&\stackrel{\B}{<} \alpha_1 Q_{\Phi_1} + \alpha_2 Q_{\Phi_2}  -\frac{d}{dt}Q_\Gamma \stackrel{\B}{\leq} 0,
		\end{aligned}
	\end{equation*}
	showing asymptotic stability by Theorem~\ref{t: lyap stab}.
\end{proof}
\vspace{0.5ex}
\begin{remark}
	Note that the conditions of this theorem consist of finding scalars $\alpha_1$ and $\alpha_2$ and a polynomial matrix $\Gamma$ such that certain QDF inequalities are satisfied on a behavior. There are many methods available in the literature for verifying such inequalities on a behavior; for example, see \cite{cotroneo2000lmi} for an LMI characterization. 
	Alternatively, one can adapt \cite[Algorithm~8]{belur_algorithmic_2002} for an LMI characterization of more general types of QDF inequalities on a behavior.
	
	For both of the aforementioned methods, the size of the LMI grows at most linearly with the orders of the supply rates and LTI systems. 
	\label{main thm remark}
\end{remark}
\subsection{Examples}Theorem~\ref{t: main diss inter thm} is an even more general dissipativity interconnection theorem than Theorem~\ref{t: first diss inter thm}. In this subsection, we show how Theorem~\ref{t: first diss inter thm} can be recovered from Theorem~\ref{t: main diss inter thm}, and how Theorem~\ref{t: main diss inter thm} can be applied to negative imaginary systems. 
\begin{example}[Theorem~\ref{t: first diss inter thm}]
	Let $\B_1,\B_2 \in \mathcal{L}^q$ and consider $\Phi_1, \Phi_2 \in \R_s^{q \times q}[\zeta,\eta]$. Define $\B := \B_1 \cap \B_2$. Under the assumptions of Theorem~\ref{t: first diss inter thm}, we have that \eqref{main thm cond 1} and \eqref{main thm cond non strict} hold with $\Gamma=0$. Hence, the statements \ref{first stat first diss inter thm} and \ref{second stat first diss inter thm} of Theorem~\ref{t: first diss inter thm} are special cases of statements \ref{main thm stat 1} and \ref{main thm stat 2} of Theorem~\ref{t: main diss inter thm}, respectively.
		\label{ex. prev thm}
\end{example}

\begin{example}[Negative imaginary systems]
	Consider the negative imaginary behaviors $\bar{\B}_1, \bar{\B}_2\in \mathcal{L}^q$ and their interconnection $\B =\B_1 \cap \B_2$ as in Example \ref{ex: NI}. 
	Recall the definition of $\Phi_i$ and $\Psi_i$ for $i = 1,2$ from this example as well. Then, $Q_{\Psi_i}$ is a storage function for $(\B_i,Q_{\Phi_i})$ for $i=1,2$. Moreover, for $\alpha_1=\alpha_2=1$,
	\vspace{-0.5ex} \begin{equation*}\alpha_1\Phi_1(\zeta,\eta)+\alpha_1\Phi_2(\zeta,\eta)=\!(\zeta+\eta)\frac{1}{2}\!\begin{bmatrix}
			0 & I \\
			I & 0
		\end{bmatrix}.\vspace{-0.5ex}\end{equation*} Then for
	\vspace{-0.5ex} \[\Gamma(\zeta,\eta):=\frac{1}{2}\begin{bmatrix}
		0 & I \\
		I & 0
	\end{bmatrix},\] \eqref{main thm cond 1} holds. Thus, using Theorem~\ref{t: main diss inter thm}, $\B$ is stable if 
	\vspace{-0.5ex}
	\begin{equation}
		Q_{\Psi_1}+Q_{\Psi_2}-Q_\Gamma \stackrel{\B}{>} 0.
		\label{NI cond}
	\end{equation}
	Therefore, unlike Theorem~\ref{t: first diss inter thm}, Theorem~\ref{t: main diss inter thm} is applicable to the case of negative imaginary systems. 

	We verify the condition 
	\eqref{NI cond} for the behaviors $\bar{\B}_i = \ker R_i$ with $R_i$ as in \eqref{Ris of NI systems}, for $i=1,2$, by solving the associated LMIs (see Remark \ref{main thm remark}). The following are the coefficient matrices $\tilde{\Psi}_1,\tilde{\Psi}_2$ of the storage functions $Q_{\Psi_1},Q_{\Psi_2}$ that satisfy them:
	\begin{equation*}
		\resizebox{\hsize}{!}{$%
			\tilde{\Psi}_1=
			\begin{bmatrix}
				\begin{bmatrix}
					0.5 & 0 \\ 0 & 0
				\end{bmatrix} & 0 \\ 0 & 0_{8\times 8}
			\end{bmatrix}
			, \;
			\tilde{\Psi}_2=\begin{bmatrix}
				0.07 & -0.30 &  0.00 & -0.16 & -0.03 &  0.08 & -0.01 &  0.11 &  0.00 &  0.03 \\
				-0.30 &  3.73 &  0.06 & -0.19 &  0.26 & -1.39 &  0.04 & -0.93 & -0.02 & -0.29 \\
				0.00 &  0.06 &  0.09 & -0.38 & -0.03 & -0.08 & -0.01 &  0.01 &  0.00 &  0.03 \\
				-0.16 & -0.19 & -0.38 &  1.60 &  0.22 & -0.12 &  0.12 & -0.61 & -0.06 & -0.31 \\
				-0.03 &  0.26 & -0.03 &  0.22 &  0.02 & -0.05 & -0.01 &  0.01 &  0.01 & -0.01 \\
				0.08 & -1.39 & -0.08 & -0.12 & -0.05 &  0.83 &  0.12 & -0.08 & -0.06 & -0.03 \\
				-0.01 &  0.04 & -0.01 &  0.12 & -0.01 &  0.12 &  0.00 &  0.05 &  0.00 &  0.01 \\
				0.11 & -0.93 &  0.01 & -0.61 &  0.01 & -0.08 &  0.05 & -0.17 & -0.03 & -0.04 \\
				0.00 & -0.02 &  0.00 & -0.06 &  0.01 & -0.06 &  0.00 & -0.03 &  0.00 & -0.01 \\
				0.03 & -0.29 &  0.03 & -0.31 & -0.01 & -0.03 &  0.01 & -0.04 & -0.01 &  0.00
			\end{bmatrix}.
			$%
		}%
	\end{equation*}
	Thus, we confirm that $\B$ is stable. 
	
	We note that one cannot directly use our theorem to conclude asymptotic stability, as the notion of strictly negative imaginary transfer functions in \cite{xiong_negative_2010} is not equivalent to negative imaginary behaviors with strict dissipation inequalities. For example, $\bar{\B}_1$ was shown in \cite{petersen_negative_2016} to have a strictly negative imaginary transfer function. However, there does not exists $Q_{\bar{\Psi}_1}$ such that
	$\frac{d}{dt}Q_{\bar{\Psi}_1}(w)< Q_{\bar{\Phi}_1}(w)= \dot{y}^\top u,$
	for all $w=(u,y)\in \bar{\B}_1$. 
	One can easily see this from the fact that $\bar{\B}_1$ contains constant trajectories $u=y=c\neq 0$ for which both sides of the inequality above are zero. 
	
	Nonetheless, we have that the only constant trajectory in the interconnected behavior $\B$ is the zero trajectory. 
	Consequently, $Q_{\Psi_1}$ satisfies its dissipation inequality strictly on $\B$.
	Then, the Lyapunov function for $\B$ as in \eqref{main thm lyap} with ${\alpha_1=\alpha_2=1}$ is such that $Q_\Psi \stackrel{\B}{\geq}0$ and 
	\[\frac{d}{dt}Q_\Psi =\frac{d}{dt}Q_{\Psi_1}+\frac{d}{dt}Q_{\Psi_2}-\frac{d}{dt}Q_\Gamma \stackrel{\B}{<}Q_{\Phi_1}+Q_{\Phi_2}  -\frac{d}{dt}Q_\Gamma =0. \]
	Hence, in agreement with \cite{petersen_negative_2016}, we can conclude that the interconnection is asymptotically stable. 
	
	Theorem~\ref{t: main diss inter thm}, applied to negative imaginary systems, is related to other results in the literature on this topic \cite{xiong_negative_2010,meskin_generalized_2024}. The main difference is that we do not necessarily require the storage functions to be positive. Moreover, in \cite[Theorem~1]{xiong_negative_2010}, the authors have shown that a DC gain condition on the systems is necessary and sufficient for asymptotic stability.  In our future work, we will explore the connection between this condition and the condition \eqref{main thm cond non strict} together with the exclusion of nonzero constant trajectories in the interconnected behavior.
\end{example}

\section{Conclusions and future work}
\label{sec:conclusion}
In this paper, we have proven two interconnection stability theorems for dissipative systems. We have leveraged the behavioral approach and, in particular, QDFs to capture general dynamic supply rates. Our first result, Theorem~\ref{t: first diss inter thm}, establishes stability of an interconnected system by means of a Lyapunov function that is the (weighted) sum of storage functions. This result generalizes several existing theorems from the literature, including the small gain and passivity theorems, and results for delta dissipative systems.

The main contribution of this paper is Theorem~\ref{t: main diss inter thm}, which establishes more general conditions for (asymptotic) stability of various interconnections between dissipative systems with dynamic supply rates. It introduces a new class of Lyapunov functions that are the sum of (indefinite) storage functions and a so-called coupling QDF, which turned out to be crucial for certain dynamic supply rates, like those for negative imaginary systems. 

We note that the stability of interconnections can also be studied using integral quadratic constraints (IQCs) \cite{megretski_system_1997}. As already discussed in \cite{seiler_stability_2015,scherer_dissipativity_2022}, dissipativity-based analysis and IQCs are strongly related. Differences between our work and the main IQC theorem \cite[Theorem~1]{megretski_system_1997} lie in the use of a coupling QDF. Our conditions are different in nature, as they do not solely consist of conditions on (dissipative) properties of the individual sub-systems, but also depend on their behavioral coupling. 
An advantage of using the coupling QDF is that it allows us to conclude on the stability of interconnections of negative imaginary systems, which, as was explained in \cite{lanzon_stability_2008}, cannot be done using IQCs.

It is of interest to study the extension of our work to interconnections of more than two behaviors, as well as the study of partial, rather than full, interconnections. In our future work, we will also explore a generalization towards interconnections of a dissipative LTI behavior and a possibly nonlinear and time-varying system.
A similar problem was studied in \cite{takaba_stability_2005,willems_dissipativity_2007} for the case of feedback interconnections, where Lyapunov functions consisting of a sum of storage functions were used. We will study how to generalize their work using our extended class of Lyapunov functions.

Lastly, there is growing interest in the discrete-time counterpart of QDFs, quadratic difference forms, in particular, in data-driven control \cite{van_waarde_behavioral_2024,maupong_lyapunov_2017,rosa_one-shot_2021}. We will explore a discrete-time version of our results that will allow for the data-driven stability analysis of system interconnections.

\appendix
The following technical lemma shows under which conditions dissipativity properties of a given behavior also hold on an associated auxiliary behavior.
\begin{lemma}
	Let $\bar{\mathcal{B}} \in \mathcal{L}^q$ and $\bar{\Phi} \in \R_s^{q \times q}[\zeta,\eta]$. Suppose that $Q_{\bar{\Psi}}$ is a storage function for $(\bar{\B},Q_{\bar{\Phi}})$. Let $N\in \R^{r \times q}$ and define $\B:=\set{w}{Nw \in \bar{\B}}$. Moreover, define
	\vspace{-0.5ex}
	\begin{equation*}
		\Phi(\zeta,\eta) := N^\top \bar{\Phi}(\zeta,\eta)N \quad \text{and} \quad \Psi(\zeta,\eta) := N^\top \bar{\Psi}(\zeta,\eta)N.
	\end{equation*}
	The following statements hold.
	\begin{enumerate}[label=\roman*)]
		\item The QDF $Q_\Psi$ is a storage function for $(\B,Q_{\Phi})$. \label{lemma stat 1}
		\item If $N$ has full column rank, then 
		\[\frac{d}{dt}Q_{\bar{\Psi}} \stackrel{\bar{\B}}{<}Q_{\bar{\Phi}} \implies \frac{d}{dt}Q_\Psi \stackrel{\B}{<}Q_\Phi. \vspace{-1.5ex}\] \label{lemma stat 2}
		\item If $Q_{\bar{\Psi}} \stackrel{\bar{\B}}{\geq}0$, then $Q_\Psi \stackrel{\B}{\geq}0$. \label{lemma stat 3}
		\item If $Q_{\bar{\Psi}} \stackrel{\bar{\B}}{>}0$ and $N$ has full column rank, then $Q_\Psi \stackrel{\B}{>}0$. \label{lemma stat 4}
	\end{enumerate}
	\label{lemma barB dissipativity}
\end{lemma}
\begin{proof}
	First, we show \ref{lemma stat 1}. We have that 
	\begin{equation}\frac{d}{dt}Q_{\bar{\Psi}} \stackrel{\bar{\B}}{\leq}Q_{\bar{\Phi}}.\label{lemma p diss ineq}\end{equation} Then, for all $w\in \B$ we have that 
	\begin{equation}\frac{d}{dt}Q_{\Psi}(w) =\frac{d}{dt}Q_{\bar{\Psi}}(Nw) \leq Q_{\bar{\Phi}}(Nw)=Q_\Phi(w). \label{bar B storage lemma ineq}\end{equation}
	Hence, $Q_{\Psi}$ is a storage function for $(\B,Q_{\Phi})$.
	
	Now, we show \ref{lemma stat 2}. Let $w\in \B$. As $Q_{\bar{\Psi}}$ satisfies the dissipation inequality \eqref{lemma p diss ineq} strictly, \eqref{bar B storage lemma ineq} holds and
	\begin{equation*}\frac{d}{dt}Q_{\Psi}(w) =Q_{\Phi}(w)\end{equation*} 
	if and only if $Nw=0$. As $N$ has full column rank, this is equivalent to $w=0$. Hence, $Q_{\Psi}$ satisfies 
	\begin{equation*}\frac{d}{dt}Q_{\Psi} \stackrel{\B}{<}Q_{\Phi}. \vspace{-0.5ex}\end{equation*}
	
	To show \ref{lemma stat 3} suppose that $Q_{\bar{\Psi}} \stackrel{\bar{\B}}{\geq}0$. Then, for all $w\in \B$, $Q_{\Psi}(w)=Q_{\bar{\Psi}}(Nw)\geq 0$. Hence, $Q_{\Psi} \stackrel{\B}{\geq}0$. 
	
	Lastly, suppose that $Q_{\bar{\Psi}} \stackrel{\bar{\B}}{>}0$ and that $N$ has full column rank. Let $w\in \B$. Then, as was shown for \ref{lemma stat 3}, $Q_{\Psi}\stackrel{\B}{\geq}0$, and $Q_{\Psi}(w)=0$ if and only if $Nw=0$, which is equivalent to $w=0$. Hence, $Q_{\Psi} \stackrel{\B}{>}0$. This shows \ref{lemma stat 4}.
\end{proof}
\bibliographystyle{IEEEtran}
\bibliography{references}

@article{belur_algorithmic_2002,
	title = {Algorithmic {Issues} in the {Synthesis} of {Dissipative} {Systems}},
	volume = {8},
	issn = {1387-3954},
	doi = {10.1076/mcmd.8.4.407.15848},
	number = {4},
	journal = {Mathematical and Computer Modelling of Dynamical Systems},
	publisher = {Taylor \& Francis},
	author = {Belur, Madhu N. and Trentelman, Harry L.},
	month = dec,
	year = {2002},
	pages = {407--428},
}

@article{petersen_negative_2016,
	title = {Negative imaginary systems theory and applications},
	volume = {42},
	issn = {13675788},
	doi = {10.1016/j.arcontrol.2016.09.006},
	language = {en},
	urldate = {2026-02-24},
	journal = {Annual Reviews in Control},
	author = {Petersen, Ian R.},
	year = {2016},
	pages = {309--318},
}

@inproceedings{meskin_generalized_2024,
	title = {Generalized {Dissipativity} and {Nonlinear} {Negative} {Imaginary} {Systems}},
	issn = {2766-6522},
	doi = {10.1109/CONTROL60310.2024.10531998},
	booktitle = {{UKACC} {International} {Conference} on {Control}},
	author = {Meskin, Nader and Mabrok, Mohamed A. and Lanzon, Alexander},
	month = apr,
	year = {2024},
	pages = {137--142},
}

@book{desoer_feedback_2009,
	series = {Classics in {Applied} {Mathematics}},
	title = {Feedback {Systems}},
	isbn = {978-0-89871-670-2},
	doi = {10.1137/1.9780898719055},
	urldate = {2026-02-17},
	publisher = {Society for Industrial and Applied Mathematics},
	author = {Desoer, Charles A. and Vidyasagar, M.},
	month = jan,
	year = {2009},
}

@article{schweidel_compositional_2022,
	title = {Compositional {Analysis} of {Interconnected} {Systems} {Using} {Delta} {Dissipativity}},
	volume = {6},
	issn = {2475-1456},
	doi = {10.1109/LCSYS.2021.3084974},
	journal = {IEEE Control Systems Letters},
	author = {Schweidel, Katherine S. and Arcak, Murat},
	year = {2022},
	pages = {662--667},
}

@incollection{arcak_stability_2016,
	title = {Stability of {Interconnected} {Systems}},
	isbn = {978-3-319-29928-0},
	doi = {10.1007/978-3-319-29928-0_2},
	booktitle = {Networks of {Dissipative} {Systems}: {Compositional} {Certification} of {Stability}, {Performance}, and {Safety}},
	publisher = {Springer International Publishing},
	author = {Arcak, Murat and Meissen, Chris and Packard, Andrew},
	year = {2016},
	pages = {13--21},
}

@article{willems_dissipative_1972,
	title = {Dissipative dynamical systems part {I}: {General} theory},
	volume = {45},
	issn = {1432-0673},
	shorttitle = {Dissipative dynamical systems part {I}},
	doi = {10.1007/BF00276493},
	language = {en},
	number = {5},
	journal = {Archive for Rational Mechanics and Analysis},
	author = {Willems, Jan C.},
	month = jan,
	year = {1972},
	pages = {321--351},
}

@article{willems_dissipative_1972-1,
	title = {Dissipative dynamical systems {Part} {II}: {Linear} systems with  supply rates},
	volume = {45},
	issn = {1432-0673},
	shorttitle = {Dissipative dynamical systems {Part} {II}},
	doi = {10.1007/BF00276494},
	language = {en},
	number = {5},
	journal = {Archive for Rational Mechanics and Analysis},
	author = {Willems, Jan C.},
	month = jan,
	year = {1972},
	pages = {352--393},
}

@book{van_der_schaft_dissipative_2017,
	title = {L2-{Gain} and {Passivity} {Techniques} in {Nonlinear} {Control}},
	publisher = {Springer International Publishing},
	author = {van der Schaft, Arjan},
	year = {2017}
}

@book{polderman_introduction_1998,
	title = {Introduction to {Mathematical} {Systems} {Theory}},
	isbn = {978-1-4757-2955-9 978-1-4757-2953-5},
	publisher = {Springer},
	author = {Polderman, Jan Willem and Willems, Jan C.},
	year = {1998},
	doi = {10.1007/978-1-4757-2953-5},
}

@article{willems_quadratic_1998,
	title = {On {Quadratic} {Differential} {Forms}},
	volume = {36},
	issn = {0363-0129},
	doi = {10.1137/S0363012996303062},
	number = {5},
	journal = {SIAM Journal on Control and Optimization},
	author = {Willems, J. C. and Trentelman, H. L.},
	month = jan,
	year = {1998},
	pages = {1703--1749},
}

@article{lanzon_stability_2008,
	title = {Stability {Robustness} of a {Feedback} {Interconnection} of {Systems} {With} {Negative} {Imaginary} {Frequency} {Response}},
	volume = {53},
	issn = {1558-2523},
	doi = {10.1109/TAC.2008.919567},
	number = {4},
	urldate = {2025-11-04},
	journal = {IEEE Transactions on Automatic Control},
	author = {Lanzon, Alexander and Petersen, Ian R.},
	month = may,
	year = {2008},
	pages = {1042--1046},
}

@article{trentelman_every_1997,
	series = {System and {Control} {Theory} in the {Behavioral} {Framework}},
	title = {Every storage function is a state function},
	volume = {32},
	issn = {0167-6911},
	doi = {10.1016/S0167-6911(97)00081-9},
	number = {5},
	journal = {Systems \& Control Letters},
	author = {Trentelman, H. L. and Willems, J. C.},
	month = dec,
	year = {1997},
	pages = {249--259},
}

@article{xiong_negative_2010,
	title = {A {Negative} {Imaginary} {Lemma} and the {Stability} of {Interconnections} of {Linear} {Negative} {Imaginary} {Systems}},
	volume = {55},
	issn = {1558-2523},
	doi = {10.1109/TAC.2010.2052711},
	number = {10},
	journal = {IEEE Transactions on Automatic Control},
	author = {Xiong, Junlin and Petersen, Ian R. and Lanzon, Alexander},
	month = oct,
	year = {2010},
	pages = {2342--2347},
}

@article{willems_interconnections_1997,
	title = {On interconnections, control, and feedback},
	volume = {42},
	issn = {1558-2523},
	doi = {10.1109/9.557576},
	number = {3},
	journal = {IEEE Transactions on Automatic Control},
	author = {Willems, Jan C.},
	month = mar,
	year = {1997},
	pages = {326--339},
}

@article{willems_time_1986I,
	title = {From time series to linear system—{Part} {I}. {Finite} dimensional linear time invariant systems},
	volume = {22},
	issn = {0005-1098},
	doi = {10.1016/0005-1098(86)90066-X},
	number = {5},
	journal = {Automatica},
	author = {Willems, Jan C.},
	month = sep,
	year = {1986},
	pages = {561--580},
}

@article{willems_time_1986II,
	title = {From time series to linear system—{Part} {II}. {Exact} modelling},
	volume = {22},
	issn = {0005-1098},
	doi = {10.1016/0005-1098(86)90005-1},
	number = {6},
	journal = {Automatica},
	author = {Willems, Jan C.},
	month = nov,
	year = {1986},
	pages = {675--694},
}

@article{takaba_stability_2005,
	title = {Stability of {Dissipative} {Interconnections}},
	volume = {40},
	issn = {2188-5079},
	language = {ja},
	number = {3-1-1-6},
	journal = {IEICE Proceedings Series},
	publisher = {The Institute of Electronics, Information and Communication Engineers},
	author = {Takaba, Kiyotsugu and Willems, Jan C.},
	month = oct,
	year = {2005},
}

@article{zames_input-output_1966,
	title = {On the input-output stability of time-varying nonlinear feedback systems {Part} one: {Conditions} derived using concepts of loop gain, conicity, and positivity},
	volume = {11},
	issn = {1558-2523},
	shorttitle = {On the input-output stability of time-varying nonlinear feedback systems {Part} one},
	doi = {10.1109/TAC.1966.1098316},
	number = {2},
	journal = {IEEE Transactions on Automatic Control},
	author = {Zames, G.},
	month = apr,
	year = {1966},
	pages = {228--238},
}

@article{willems_dissipativity_2007,
	title = {Dissipativity and stability of interconnections},
	volume = {17},
	copyright = {Copyright © 2006 John Wiley \& Sons, Ltd.},
	issn = {1099-1239},
	doi = {10.1002/rnc.1121},
	language = {en},
	number = {5-6},
	urldate = {2026-03-20},
	journal = {International Journal of Robust and Nonlinear Control},
	author = {Willems, Jan C. and Takaba, Kiyotsugu},
	year = {2007},
	pages = {563--586},
}

@article{maupong_lyapunov_2017,
	series = {20th {IFAC} {World} {Congress}},
	title = {On {Lyapunov} functions and data-driven dissipativity},
	volume = {50},
	issn = {2405-8963},
	doi = {10.1016/j.ifacol.2017.08.1052},
	number = {1},
	journal = {IFAC-PapersOnLine},
	author = {Maupong, T. M. and Mayo-Maldonado, J. C. and Rapisarda, P.},
	month = jul,
	year = {2017},
	pages = {7783--7788},
}

@article{van_waarde_behavioral_2024,
	title = {A {Behavioral} {Approach} to {Data}-{Driven} {Control} {With} {Noisy} {Input}–{Output} {Data}},
	volume = {69},
	issn = {1558-2523},
	doi = {10.1109/TAC.2023.3275014},
	number = {2},
	journal = {IEEE Transactions on Automatic Control},
	author = {van Waarde, Henk J. and Eising, Jaap and Camlibel, M. Kanat and Trentelman, Harry L.},
	month = feb,
	year = {2024},
	pages = {813--827},
}

@inproceedings{rosa_one-shot_2021,
	title = {On the one-shot data-driven verification of dissipativity of {LTI} systems with general quadratic supply rate function},
	doi = {10.23919/ECC54610.2021.9654933},
	booktitle = {{European} {Control} {Conference}},
	author = {Rosa, Tábitha E. and Jayawardhana, Bayu},
	month = jun,
	year = {2021},
	pages = {1291--1296},
}

@article{scherer_dissipativity_2022,
	title = {Dissipativity and {Integral} {Quadratic} {Constraints}: {Tailored} {Computational} {Robustness} {Tests} for {Complex} {Interconnections}},
	volume = {42},
	issn = {1941-000X},
	doi = {10.1109/MCS.2022.3157117},
	number = {3},
	journal = {IEEE Control Systems Magazine},
	author = {Scherer, Carsten W.},
	month = jun,
	year = {2022},
	pages = {115--139},
}

@article{seiler_stability_2015,
	title = {Stability {Analysis} {With} {Dissipation} {Inequalities} and {Integral} {Quadratic} {Constraints}},
	volume = {60},
	issn = {1558-2523},
	doi = {10.1109/TAC.2014.2361004},
	number = {6},
	journal = {IEEE Transactions on Automatic Control},
	author = {Seiler, Peter},
	month = jun,
	year = {2015},
	pages = {1704--1709},
}

@article{megretski_system_1997,
	title = {System {Analysis} via {Integral} {Quadratic} {Constraints}},
	author = {Megretski, Alexandre and Rantzer, Anders},
	journal={IEEE transactions on automatic control},
	volume={42},
	number={6},
	pages={819--830},
	year={1997},
	publisher={IEEE}
}

@book{zhou_essentials_1998,
	series = {Prentice {Hall} {Modular} {Series} for {Eng}},
	title = {Essentials of {Robust} {Control}},
	isbn = {978-0-13-525833-0},
	publisher = {Prentice Hall},
	author = {Zhou, K. and Doyle, J.C.},
	year = {1998},
	lccn = {97029568},
}

@inproceedings{cotroneo2000lmi,
	title={An LMI condition for nonnegativity of a Quadratic Differential Form along a Behavior},
	author={Cotroneo, T and Willems, JC},
	booktitle={Proc. of 14th Int. Symp. on Math. Theory of Netw. Syst.(MTNS2000)},
	year={2000}
}

\end{document}